\documentclass[]{article}

\usepackage{amsfonts,amssymb,amsbsy,latexsym,amsmath,tabulary,
graphicx,times,caption,fancyhdr,yfonts}
\usepackage[utf8]{inputenc}
\usepackage{url,multirow,morefloats,floatflt,cancel,textcomp,tfrupee}
\usepackage{pifont}
\usepackage[nointegrals]{wasysym}
\usepackage{pifont}

\newtheorem{theorem}{Theorem}[section]
\newtheorem{definition}[theorem]{Definition}

\newtheorem{proposition}[theorem]{Proposition}

\newtheorem{remark}[theorem]{Remark}

\newenvironment{proof}{\noindent\mbox{\bf Proof.}}
{\hfill\mbox{\ding{113}}\bigskip}

\makeatletter

\AtBeginDocument{
\expandafter\ifx\csname eqalign\endcsname\relax
\def\eqalign#1{\null\vcenter{\def\\{\cr}\openup\jot\m@th
  \ialign{\strut$\displaystyle{##}$\hfil&$\displaystyle{{}##}$\hfil
      \crcr#1\crcr}}\,}
\fi
}

\let\lt=<
\let\gt=>
\def\processVert{\ifmmode|\else\textbar\fi}

\@ifundefined{subparagraph}{
\def\subparagraph{\@startsection{paragraph}{5}{2\parindent}{0ex plus 0.1ex minus 0.1ex}%
{0ex}{\normalfont\small\itshape}}%
}{}

\newcommand\role[1]{\unskip}
\newcommand\aucollab[1]{\unskip}

\@ifundefined{tsGraphicsScaleX}{\gdef\tsGraphicsScaleX{1}}{}
\@ifundefined{tsGraphicsScaleY}{\gdef\tsGraphicsScaleY{.9}}{}
\def\checkGraphicsWidth{\ifdim\Gin@nat@width>\textwidth
	\tsGraphicsScaleX\textwidth\else\Gin@nat@width\fi}

\def\checkGraphicsHeight{\ifdim\Gin@nat@height>.9\textheight
	\tsGraphicsScaleY\textheight\else\Gin@nat@height\fi}

\def\fixFloatSize#1{\@ifundefined{processdelayedfloats}{\setbox0=\hbox{\includegraphics{#1}}\ifnum\wd0<\columnwidth\relax\renewenvironment{figure*}{\begin{figure}}{\end{figure}}\fi}{}}
\let\ts@includegraphics\includegraphics

\def\inlinegraphic[#1]#2{{\edef\@tempa{#1}\edef\baseline@shift{\ifx\@tempa\@empty0\else#1\fi}\edef\tempZ{\the\numexpr(\numexpr(\baseline@shift*\f@size/100))}\protect\raisebox{\tempZ pt}{\ts@includegraphics{#2}}}}

\AtBeginDocument{\def\includegraphics{\@ifnextchar[{\ts@includegraphics}{\ts@includegraphics[width=\checkGraphicsWidth,height=\checkGraphicsHeight,keepaspectratio]}}}

\def\URL#1#2{\@ifundefined{href}{#2}{\href{#1}{#2}}}

\def\UrlOrds{\do\*\do\-\do\~\do\'\do\"\do\-}%
\g@addto@macro{\UrlBreaks}{\UrlOrds}
\makeatother

\makeatletter

\def\wileyIndent{1pt}
\usepackage[paperheight=10in,paperwidth=6.5in,margin=2cm,headsep=.5cm,top=2.5cm]{geometry}

\renewenvironment{abstract}
{\vspace*{-1pc}\trivlist\item[]\leftskip\wileyIndent\hrulefill\par\vskip4pt\noindent\textbf{\abstractname}\mbox{\null}\\}{\par\noindent\hrulefill\endtrivlist}

\def\author#1{\gdef\@author{\hskip-\dimexpr(\tabcolsep)\hskip\wileyIndent\parbox{\dimexpr\textwidth-\wileyIndent}{\centering\bfseries#1}}}

\def\title#1{\gdef\@title{\centering\bfseries\ifx\@articleType\@empty\else\@articleType\\\fi#1}}

\let\@articleType\@empty \def\articletype#1{\gdef\@articleType{{\normalfont\itshape#1}}}

\fancypagestyle{headings}{\fancyhf{}\fancyhead[C]{\RunningHead}\fancyhead[R]{\thepage}}

 \def\audegree#1{}

\date{}

\makeatother

\usepackage[T1]{fontenc}
\makeatother
\usepackage[numbers,sort&compress]{natbib}

\begin{document}

\title{On Arithmetical Truth of the \\  Self-Referential Sentences}

\author{\textsc{Kaave Lajevardi \,\&\,   Saeed Salehi}}

\def\RunningHead{\textsc{\small K.~Lajevardi  \& S.~Salehi } / {\sl On Arithmetical Truth of the Self-Referential Sentences \; }}
\def\RunningAuthor{S.~Salehi \& K.~Lajevardi}


\maketitle

\begin{abstract}
We take an argument of G\"odel's from his ground-breaking 1931 paper, 
generalize it, and examine its 
validity.
The argument in question is this: {\sl the sentence G says about itself that it is not provable, and G is indeed not provable;
therefore, G is true.}
\end{abstract}

\section{Introduction}
As is well known, G\"odel begins his~1931 masterpiece \cite{godel} with an introductory section, Section~1, wherein he explains the main ideas behind the (first) incompleteness theorem in an informal and intuitive way, with an explicit caveat that those remarks are being made ``without any claim to complete precision''. What he does in that section is, among other things, to introduce, in a very lucid way, the idea of encoding the syntax and that of diagonalization.

But G\"odel's introduction is not confined to what can be found, more formally
expressed, in the technical parts of the paper. G\"odel's official statement of the first incompleteness theorem (his Theorem~VI) asserts that for every theory satisfying certain conditions, there is a sentence, now called the G\"odel sentence of the theory, which is expressible in the language of the theory but is neither provable nor refutable in the theory---the theorem does {\em not} talk about the truth of that undecidable sentence. In his Section~1, however, G\"odel chooses to also talk about the truth of the G\"odel sentence of the system of {\em Principia Mathematica}, and observes that the proof of the first incompleteness theorem, as presented in that section, is similar to Richard's paradox. Whatever expository merits it might have, this move is a problematic one, and G\"odel shows awareness of this. He says that his method of proof is applicable to any formal system that has two conditions, the second of which being ``every provable formula is {\em true} in the interpretation considered'' (\cite[page~151]{godel}  emphasis added). He then writes, ``The purpose of carrying out the above proof with full precision in what follows is, among other things, {\em to replace the second of the assumptions just mentioned by a purely formal and much weaker one}'' (ibid,   emphasis added).
This ``much weaker'' condition, called $\omega$-consistency by
G\"odel, is one of the points of focus in our discussion (note that $\omega$-consistency is stronger than [simple] consistency, see e.g. \cite[Theorems~14,15]{isaac}).

Our purpose in this note is not to philosophize about the notion of truth or doing exegetical work on its r\^{o}le in G\"odel's classic paper. Rather, we wish to draw attentions to an informal argument, to the effect that the G\"odel sentence of the system {\em Principia Mathematica} is true (that is, true in $\mathbb{N}$), which is presented by G\"odel in his introductory section.
The argument in question goes as follows \cite[p.~151]{godel} (italics in the original):

\bigskip

\begin{tabular}{r}
\qquad \qquad \qquad From the remark that $[R(q);q]$ says about itself  that it \\
\qquad \qquad \qquad   is not provable,   it follows at once that $[R(q);q]$ is true, \\
\qquad \qquad \qquad for  $[R(q);q]$ {\em is}  indeed unprovable (being undecidable).
\end{tabular}

\bigskip

\noindent Here $[R(q); q]$ is what is now called the {\em  G\"odel sentence} of a theory (or ``system'') which is subject to the first incompleteness theorem (that can be taken as Peano's Arithmetic $\textsl{\textsf{PA}}$ or its recursively axiomatizable   $\omega$-consistent extensions). This is nowadays  denoted by  $\textsl{\textsf{G}}$, which is by definition   a sentence $P$ that is  equivalent to $\neg{\sf Pr} (\#P)$ (also provably so), where ${\sf Pr}$ is the provability predicate of the theory and $\#P$ is the standard term for  the G\"odel number of $P$.

As we understand the above passage, its logical form is the following, where $A$ is an arbitrary sentence expressible in the language of the theory:

\begin{enumerate}\itemindent=3em
\item[(1)]  $A$ says about itself that it has a property $F$.
\item[(2)]  $A$ indeed has the property $F$.
\item[(3)]  Therefore: $A$ is true.
\end{enumerate}

That G\"odel says that it follows {\em at once} (G\"odel's term ``sofort'' could also be translated to {\em immediately})   that $\textsl{\textsf{G}}$ is true suggests that, in G\"odel's view, we are not dealing with an enthymeme---it seems to us that, for G\"odel, the argument scheme has no missing premises. It is our task in the next section to argue that the (1)--(3) argument scheme is {\em invalid}, that is to say, there are situations wherein the premises are true while the conclusion false.
Naturally enough, the validity of the argument hinges on its terms, in particular on what it is for a sentence to be ``true'', and what is meant by a sentence ``saying of itself'' that it has a certain property. As for the first term,   it is almost obvious from G\"odel's introductory section that, like many  modern writers in mathematical logic, when he writes ``true'' simpliciter, he means {\em true in the standard model}, $\mathbb{N}$ (thus \cite{godel}'s footnote~4 on page~145: ``... no other notions occur but $+$~(addition) and $\cdot$~(multiplication), both for natural numbers, and in which the quantifiers $(x)$, too, apply to natural numbers only.'')   Regarding the notion of saying something of oneself, we shall consider two  interpretations that might be ascribed to G\"odel.

\section{The Invalidity of the Generalized Argument}
How are we to understand the expression  ``$A$ says that it has the property $F$''?
Modulo an agreement over the meaning of ``holding'', to which we shall return shortly,  we find it quite plausible to think that if the conditional
$A\longrightarrow F(\# A)$ holds, then $A$ says, {\em inter alia}, that $A$ has property $F$.
So, if we take ``holding'' as ``being true in the standard model of natural numbers $\mathbb{N}$'' and   ``saying'' as ``implying'', then the argument is not valid: if we take  $\gamma$ to be the sentence $\textsl{\textsf{G}}\wedge(0\!=\!1)$ then  both of the sentences $\gamma\rightarrow\neg{\sf Pr}(\#\gamma)$ and $\neg{\sf Pr}(\#\gamma)$ hold, but $\gamma$ does not hold.
Thus to say that $A$ says, {\em exactly}, that $A$ has property $F$ must be something stronger, to wit that the biconditional  $A\longleftrightarrow F(\# A)$ holds. And this is, in fact, what is taken by some authors (including G\"odel)  to be the meaning of
``$A$ says of $A$ that it has property $F$'' (see e.g. \cite{hv1,milne} and references therein).

As for the meaning of ``holding'', we already presented evidence
for the claim that what G\"odel meant by it is being true in the standard model,
$\mathbb{N}$. However, let us recognize another reading of it---which is also
suggested in the literature (see \cite{milne,hv1} and  references therein)---namely
{\em being provable in a given theory}.
Having fixed  a theory $T$
(which we suppose to be an $\omega$-consistent, recursively axiomatizable extension of Peano's Arithmetic $\textsl{\textsf{PA}}$), we then have eight possible ways of interpreting the (1)--(3) argument scheme. Here is the complete list, of which we find (VII), (V), and (III) the most interesting:

\smallskip

$${\rm (I)}\frac{\;\mathbb{N}\vDash A\leftrightarrow F(\# A), \quad \mathbb{N}\vDash F(\# A)\;}{\mathbb{N}\vDash A} \qquad {\rm (II)}\frac{\;\mathbb{N}\vDash A\leftrightarrow F(\# A), \quad \mathbb{N}\vDash F(\# A)\;}{T\vdash A}$$

$${\rm (III)}\frac{\;\mathbb{N}\vDash A\leftrightarrow F(\# A),\quad T\vdash F(\# A)\;}{\mathbb{N}\vDash A} \qquad {\rm (IV)}\frac{\;\mathbb{N}\vDash A\leftrightarrow F(\# A),\quad T\vdash F(\# A)\;}{T\vdash  A}$$

$${\rm (V)}\frac{\;T\vdash  A\leftrightarrow F(\# A),\quad \mathbb{N}\vDash  F(\# A)\;}{\mathbb{N}\vDash A} \qquad {\rm (VI)}\frac{\;T\vdash  A\leftrightarrow F(\# A),\quad \mathbb{N}\vDash  F(\# A)\;}{T\vdash   A}$$

$${\rm (VII)}\frac{\;T\vdash  A\leftrightarrow F(\# A),\quad T\vdash  F(\# A)\;}{\mathbb{N}\vDash A}  \qquad {\rm (VIII)}\frac{\;T\vdash  A\leftrightarrow F(\# A),\quad T\vdash   F(\# A)\;}{T\vdash   A} $$

\bigskip

Of these, (I) and (VIII) are of course  valid because of the truth-condition of the material conditional and Modus Ponens, respectively. For all other cases, we will present  triples $(A, F, T)$ which invalidate them.

\begin{theorem}\label{th1}
 The argument  {\rm (IV)} is invalid for  $A=\textsl{\textsf{G}}$,   $F(x)\equiv(x=\#\textsl{\textsf{G}})$, and $T=\textsl{\textsf{PA}}$.
\end{theorem}
\begin{proof}
Obviously, $\mathbb{N}\vDash F(\#$\textsl{\textsf{G}}$)$ and  $T\vdash F(\#$\textsl{\textsf{G}}$)$.
 By G\"odel's proof $\mathbb{N}\vDash\textsl{\textsf{G}}$ holds and  so  $\mathbb{N}\vDash\textsl{\textsf{G}}
\leftrightarrow F(\#\textsl{\textsf{G}})$. On the other hand, by G\"odel's theorem, $T\nvdash\textsl{\textsf{G}}$.
\end{proof}

\begin{theorem}\label{th2}
The arguments {\rm (II)} and {\rm (VI)} are invalid for  $A=\textsl{\textsf{G}}$, $F(x)\equiv~\neg{\sf Pr}(x)$, and $T=\textsl{\textsf{PA}}$.
\end{theorem}
\begin{proof}
We already have  $\mathbb{N}\vDash\textsl{\textsf{G}}\leftrightarrow
\neg{\sf Pr}(\#\textsl{\textsf{G}})$. Since by G\"odel's theorem we have $\textsl{\textsf{PA}}\nvdash\textsl{\textsf{G}}$, it follows that $\neg{\sf Pr}(\#\textsl{\textsf{G}})$ is true, i.e.,  $\mathbb{N}\vDash\neg{\sf Pr}(\#\textsl{\textsf{G}})$.
\end{proof}

Of course the arguments (III), (V) and (VII) are all valid if $T$ is a sound theory (i.e., when $\mathbb{N}\vDash T$). As we mentioned in the Introduction, G\"odel replaces the soundness condition with the much weaker condition of $\omega$-consistency. As Isaacson mentions in \cite{isaac},  G\"odel states in \cite{godel} that the notion of $\omega$-consistency is ``much weaker'' than  soundness but gives no argument for this claim.
It is shown in \cite[Proposition~19]{isaac} (with a proof attributed to Kreisel in the 1950s) that there exists a false sentence $\textsl{\textsf{K}}$ such that the theory $\textsl{\textsf{PA}}+\textsl{\textsf{K}}$ is $\omega$-consistent. Also, the sentence $\textsl{\textsf{K}}$ is a diagonal sentence of a formula $H(x)$; i.e., the equivalence $\textsl{\textsf{K}}\leftrightarrow H(\#\textsl{\textsf{K}})$ holds (is $\textsl{\textsf{PA}}$-provable and true in $\mathbb{N}$).

\begin{theorem}\label{th3}
The argument {\rm (V)} does not hold  for $A=\textsl{\textsf{K}}$,   $F(x)\equiv(x=\#\textsl{\textsf{K}})$, and $T=\textsl{\textsf{PA}}+\textsl{\textsf{K}}$.
\end{theorem}
\begin{proof}
By   $\textsl{\textsf{PA}}\vdash F(\#\textsl{\textsf{K}})$ we have $\textsl{\textsf{PA}}+\textsl{\textsf{K}}\vdash \textsl{\textsf{K}}\leftrightarrow F(\#\textsl{\textsf{K}})$. Now, $\mathbb{N}\vDash F(\#\textsl{\textsf{K}})$ holds trivially. By \cite[Propositions~19]{isaac}, $\mathbb{N}\nvDash\textsl{\textsf{K}}$.
\end{proof}

\begin{theorem}\label{th4}
Neither  {\rm (III)} nor {\rm (VII)}   hold for
$A=\textsl{\textsf{K}}$,   $F(x)\equiv H(x)$, and $T=\textsl{\textsf{PA}}+\textsl{\textsf{K}}$.
\end{theorem}
\begin{proof}
We already have $\textsl{\textsf{PA}}\vdash \textsl{\textsf{K}}\leftrightarrow H(\#\textsl{\textsf{K}})$ and
$\mathbb{N}\vDash\textsl{\textsf{K}}\leftrightarrow H(\#\textsl{\textsf{K}})$ by definition,  and so
$\textsl{\textsf{PA}}+\textsl{\textsf{K}}\vdash \textsl{\textsf{K}}\leftrightarrow H(\#\textsl{\textsf{K}})$ holds too. The latter also implies that $\textsl{\textsf{PA}}+\textsl{\textsf{K}}\vdash H(\#\textsl{\textsf{K}})$. Finally, by   \cite[Propositions~19]{isaac} we have  $\mathbb{N}\nvDash\textsl{\textsf{K}}$.
\end{proof}

Let us acknowledge the fact that perhaps it is only by overgeneralizing G\"odel's informal argument that we are making it invalid. Had we not abstracted from the specific properties of   $A$, $F$, and $T$, G\"odel's informal argument would be valid, even in the interesting cases of (III), (V), and (VII) (though it would lose much of its appeal). This is substantiated in the following:

\begin{proposition}\label{prop}
If $A,F$ are both $\Pi_1$ and $T$ is an $\omega$-consistent extension of   $\textsl{\textsf{PA}}$, then the arguments {\rm (III), (V),} and {\rm (VII)} are valid.
\end{proposition}
\begin{proof}
If $A,F(x)\in\Pi_1$ then  $F(\#A)$ and $A\leftrightarrow F(\#A)$ are both $\Sigma_2$. By \cite[Theorem~17]{isaac}, all the $T$-provable $\Sigma_2$-sentences are true. So, the provability of   $F(\#A)$ or $A\leftrightarrow F(\#A)$ in $T$ implies their truth. Whence, (III), (V), and (VII) all reduce to (I).
\end{proof}

\section{Back to the Truth of the G\"odel Sentence}
The sentence $\textsl{\textsf{G}}$   is called {\em the} G\"odel sentence of the theory in consideration (say, of $\textsl{\textsf{PA}}$).  One probable rationale  for this is that if $\textsl{\textsf{G}}'$ is any other sentence that is equivalent to its unprovability, that is $\textsl{\textsf{G}}'$ is equivalent to $\neg{\sf Pr}(\#\textsl{\textsf{G}}')$,  then $\textsl{\textsf{G}}$ and $\textsl{\textsf{G}}'$ are equivalent (see e.g., \cite{lind}). But this is not a  convincing reason when the theory is not sound:
\begin{remark}\label{rem}{\rm
Let $S=\textsl{\textsf{PA}}+\neg{\sf Con}(\textsl{\textsf{PA}})$, where ${\sf Con}(\textsl{\textsf{PA}})$ is the consistency statement of $\textsl{\textsf{PA}}$, $\neg{\sf Pr}(\#[0\!=\!1])$. By G\"odel's second incompleteness theorem the theory $S$ is consistent (but not sound). We show that for any true $\Sigma_1$-sentence $\sigma$, the sentence $\sigma'$ which is defined to be ${\sf Con}(S)\wedge\sigma$, is equivalent to its unprovability in $S$; later we also show that this  holds for any false $\Pi_1$-sentence $\varrho$ too.

To see this take $\sigma$ to be any true $\Sigma_1$-sentence. Then by L\"ob's Theorem  we have  $S\vdash{\sf Con}(S)\rightarrow\neg{\sf Pr}_S(\#{\sf Con}(S))$ and  by the $\Sigma_1$-completeness,  $S\vdash\sigma$; so $$S\vdash{\sf Con}(S)\wedge\sigma\rightarrow\neg{\sf Pr}_S(\#[{\sf Con}(S)\wedge\sigma]).$$ On the other hand $S\vdash\neg{\sf Pr}_S(\#\psi)\rightarrow{\sf Con}(S)$, for any $\psi$, which, together with $S\vdash\sigma$, implies the derivability of the converse implication $$S\vdash\neg{\sf Pr}_S(\#[{\sf Con}(S)\wedge\sigma])\rightarrow{\sf Con}(S)\wedge\sigma.$$ Thus, $\sigma'={\sf Con}(S)\wedge\sigma$ is equivalent to its unprovability in $S$: $S\vdash\sigma'\leftrightarrow\neg{\sf Pr}_S(\#\sigma')$.

Now, take $\varrho$ to be a false $\Pi_1$-sentence. Then $\mathbb{N}\vDash\neg\varrho$ and $\neg\varrho\in\Sigma_1$; so $S\vdash\neg\varrho$. Also, $S\vdash\neg{\sf Con}(S)$ (because $S\vdash\neg{\sf Con}(\textsl{\textsf{PA}})$ and $\textsl{\textsf{PA}}\subset S$); whence $S\vdash{\sf Pr}_S(\#\theta)$ for any $\theta$ and in particular $S\vdash{\sf Pr}_S(\#\varrho)$. So, $S\vdash\varrho\leftrightarrow\neg{\sf Pr}_S(\#\varrho)$ which shows that $\varrho$ is equivalent to its unprovability in $S$.

It can be easily seen that none of the true sentences $\sigma'={\sf Con}(S)\wedge\sigma$ (where we have  $\mathbb{N}\vDash\sigma\in\Sigma_1$) is equivalent to  any of the false $\Pi_1$-sentences $\varrho$ (where  $\mathbb{N}\nvDash\varrho$), even  though they  are all provably equivalent  inside the theory $S$.
}\hfill\ding{71}
\end{remark}

Thus, a sentence which is equivalent to its unprovability inside a theory, does not deserve to be called ``{\em the} G\"odel sentence'' of that theory. A better definition  could be:
\begin{definition}\label{def}{\rm
The G\"odel sentence of a theory $T$ is a sentence $P$ such that
\newline\centerline{(i)\;$T\vdash P\leftrightarrow\neg{\sf Pr}_T(\#P)$ \; and also \;  (ii)\;$\mathbb{N}\vDash P\leftrightarrow\neg{\sf Pr}_T(\#P)$.}
All such sentences are equivalent  (in $\mathbb{N}$ and provably in $T$ if $T\supseteq\textsl{\textsf{PA}}$).
}\hfill\ding{71} \end{definition}

Now, by this definition one can show that the G\"odel sentence(s) of a theory is(are) true if and only if the theory is consistent (cf.~\cite[Theorems~10,11]{isaac}):
\begin{theorem}\label{con}
Let $P$ be the G\"odel sentence of a theory $T$. Then $P$ is true if and only if $T$ is consistent.
\end{theorem}
\begin{proof}
If $P$ is true then, by Definition~\ref{def}(ii), so is $\neg{\sf Pr}_T(\#P)$ thus $T$ is consistent; and if $T$ is consistent then it cannot prove $P$ (otherwise if $T\vdash P$ then on the one hand by Definition~\ref{def}(i) $T\vdash\neg{\sf Pr}_T(\#P)$ and on the other hand  ${\sf Pr}_T(\#P)$ is a true $\Sigma_1$-sentence and so is provable in $T$ which contradicts the consistency of $T$) and so $\neg{\sf Pr}_T(\#P)$ is true whence, by Definition~\ref{def}(ii),  $P$ is true.
\end{proof}

It has been argued in the literature (see e.g. \cite{boolos} or \cite{raatikainen2} and references therein) that since the consistency of the theory implies (and moreover is equivalent to) its G\"odel's sentence(s), even provably so inside the theory, then for ``seeing '' the truth of the G\"odel sentence(s) the consistency of the theory should be seen.
Let us note, that a sentence which {\em is
equivalent to its own provability inside the theory} is not thereby (equivalent to) the
G\"odel sentence(s) of the theory.
 In Remark~\ref{rem} we had  $S\vdash {\sf Con}(S)\leftrightarrow\varrho$ for any false $\Pi_1$-sentence $\varrho$  (and also $S\vdash\varrho\leftrightarrow\neg{\sf Pr}_S(\#\varrho)$) but ${\sf Con}(S)$ is a true $\Pi_1$-sentence while $\varrho$ is not. Our last result provides a necessary and sufficient condition for the truth of all the $\Pi_1$-sentences that are equivalent to their unprovability inside the theory.
\begin{theorem}\label{nec}
For    a recursively axiomatizable extension $T$ of $\textsl{\textsf{PA}}$, all of  the $\Pi_1$-sentences $\theta$ which
 satisfy $T\vdash\theta\leftrightarrow\neg{\sf Pr}_T(\#\theta)$ are true  if and only if $T+{\sf Con}(T)$ is consistent.
\end{theorem}
\begin{proof}
If $T+{\sf Con}(T)$ is not consistent then
 $T\vdash\neg{\sf Con}(T)$ and so,
 $T\vdash{\sf Pr}(\#\psi)$ for any $\psi$. Also, for
   any false $\Pi_1$-sentence $\varrho$ we have $T\vdash\neg\varrho$. Whence,  $T\vdash\varrho\leftrightarrow\neg{\sf Pr}_T(\#\varrho)$ holds which shows that  any false $\Pi_1$-sentence is equivalent to its unprovability in $T$. Suppose now  that a false $\Pi_1$-sentence  $\theta$ satisfies $T\vdash\theta\leftrightarrow\neg{\sf Pr}_T(\#\theta)$.  Then $\neg\theta$ is a true $\Sigma_1$-sentence, whence $T\vdash\neg\theta$. This, by $T\vdash\theta\leftrightarrow{\sf Con}(T)$, implies that $T\vdash\neg{\sf Con}(T)$ and so the theory $T+{\sf Con}(T)$ is not consistent.
\end{proof}

The consistency of $T+{\sf Con}(T)$ is a strictly stronger condition than the simple consistency of $T$ (see~\cite[Corollary~37]{isaac}), while $\omega$-consistent theories satisfy this condition (see~\cite[Theorem~36]{isaac}).
Thus if a $\Pi_1$-sentence says, inside a consistent theory, that it is not provable in that theory, this, in itself, is
no reason for believing that that sentence is true (cf. Remark~\ref{rem}) unless that theory is also consistent with its own consistency statement. Finally, let us note that the consistency of $T+{\sf Con}(T)$ is a necessary and sufficient condition for the independence (unprovability and unrefutability) of the G\"odel sentences from the theory (see~\cite[Theorem~35]{isaac}).

\section{Conclusion}
G\"odel's original argument was this: the sentence $G$ says that it is not provable in $T$, and $G$ is indeed not provable in $T$; thus $G$ is true. A necessary condition for the validity of this argument is the consistency of $T$, see Theorem~\ref{con}; note that here $G$ does not need to be a $\Pi_1$-sentence. If $G$ says (only) inside the theory that it is not $T$-provable (and does not say so in the real world, i.e., $T\vdash G\leftrightarrow\neg{\sf Pr}_T(\#G)$ but $\mathbb{N}\nvDash G\leftrightarrow\neg{\sf Pr}_T(\#G)$), then $G$ need not be true, see Remark~\ref{rem}, as $T$ might not be sound (and so $T$ could be lying about $G$ or what it says!). So, G\"odel's argument is invalid in this case, even if $G$ is a $\Pi_1$-sentence. But it is valid  if $G$ is a $\Pi_1$-sentence and if $T$ is $\omega$-consistent (as G\"odel wanted the theory to be); see Theorem~\ref{nec} and the explanations after its proof.  Over-generalizing the argument proved it to be invalid (in six out of the eight possible cases), whether we take ``saying'' to be  `inside the theory' or `in the real world' ($\mathbb{N}$), and ``holding'' as `being true' (in $\mathbb{N}$) or `being provable' (in an $\omega$-consistent, recursively axiomatizable extension of Peano's Arithmetic); see Theorems~\ref{th1},~\ref{th2},~\ref{th3} and~\ref{th4}. However, the argument remains valid for $\Pi_1$-sentences and $\Pi_1$-properties, see Proposition~\ref{prop}, as they were in G\"odel's case~\cite{godel}. Our  (I)---(VIII) rules   are overgeneralizing G\"odel's argument, since he would  not mean ``the unprovability of $G$'' to be provable in the theory, because  by the second incompleteness theorem (which was intended to be proved in the never-written second part of \cite{godel}) the consistency is not provable in the theory and so it cannot prove the unprovability of anything; however, in the rules (III), (IV), (VII) and (VIII) we assumed the $F(\#A)$ to be provable in the theory.

\bibliographystyle{wileynum}


\begin{thebibliography}{99}

\bibitem{boolos}
{\sc  George Boolos},
 On ``seeing'' the Truth of the G\"odel Sentence,
 {\em  Behavioral and Brain Sciences}~13:4 (1990) 655--656.
\textsc{doi}:~{10.1017/S0140525X00080687}.   Also reprinted in:  {\em Logic, Logic and Logic} ({\sc isbn}:~{9780674537675}), R.~Jeffrey (ed.), Harvard University Press (1999),  pp.~389--391.



\bibitem{godel}
{\sc Kurt G\"odel},
{\"{U}ber formal unentscheidbare S\"{a}tze der Principia Mathematica und verwandter Systeme,~I.},
{\em Monatshefte f\"{u}r Mathematik und Physik}~38:1  (1931) 173--198. \textsc{doi}:~{10.1007/BF01700692}.
Translated as
{``On Formally Undecidable Propositions of {\em   Principia Mathematica} and Related Systems,~I.''}, in:   {\em Kurt G\"odel Collected Works, Volume~I: Publications~1929--1936} ({\sc isbn}:~{9780195039641}), S.~Feferman et al. (eds.),  Oxford University Press (1986), pp.~135--152.


\bibitem{hv1}
{\sc Volker Halbach \& Alber Visser},
 Self-reference in Arithmetic,~I,
 {\em  The Review of Symbolic Logic}~7:4 (2014)  671--691.  \textsc{doi}:~{10.1017/S1755020314000288}.
 Self-reference in Arithmetic,~II,
 {\em  ibid}~692--712.  \textsc{doi}:~{10.1017/S175502031400029X}.


\bibitem{isaac}
{\sc Daniel Isaacson},
{``Necessary and Sufficient Conditions for Undecidability of the G\"odel Sentence and its Truth''},  in:  {\em Logic, Mathematics, Philosophy: Vintage Enthusiasms---Essays in Honour of John L. Bell}, D.~DeVidi \& M. Hallett \& P. Clarke (eds.),  Springer (2011), pp.~135--152.  \textsc{doi}:~{10.1007/978-94-007-0214-1\_7}.


\bibitem{lind}
{\sc Per Lindstr\"{o}m},
 Provability Logic---a short introduction,
 {\em  Theoria}~62:1-2 (1996) 19--61.
\textsc{doi}:~{10.1111/j.1755-2567.1996.tb00529.x}.



\bibitem{milne}
{\sc Peter Milne},
 On G\"odel Sentences and What They Say,
 {\em  Philosophia Mathematica}~15:2 (2007) 193--226.
\textsc{doi}:~{10.1093/philmat/nkm015}.


\bibitem{raatikainen2} {\sc Panu Raattkainen}, On the Philosophical Relevance of G\"{o}del's Incompleteness Theorems, {\em Revue Internationale de Philosophie}~59:4  (2005)  513--534.
{\sc url}:~\texttt{
\url{https://www.cairn.info/load_pdf.php?ID_ARTICLE=RIP_234_0513}}



\end{thebibliography}


\end{document}